\documentclass[11pt,centertags,oneside]{amsart}
\usepackage{amsmath,amstext,amsthm,amscd,typearea}
\usepackage{amssymb}
\usepackage{a4wide}
\usepackage[mathscr]{eucal}
\usepackage{mathrsfs}
\usepackage{typearea}
\usepackage{charter}
\usepackage{pdfsync}
\usepackage{url}

\usepackage{xcolor}

\usepackage{enumitem}

\usepackage[a4paper,width=16.2cm,top=3cm,bottom=3cm]{geometry}

\numberwithin{equation}{section}

\newtheorem{theorem}{Theorem}[section]
\newtheorem{definition}[theorem]{Definition}
\newtheorem{proposition}[theorem]{Proposition}
\newtheorem{corollary}[theorem]{Corollary}
\newtheorem{lemma}[theorem]{Lemma}
\newtheorem{remark}[theorem]{Remark}

\newtheorem*{theorem*}{Theorem}
\newtheorem*{main-theorem*}{Main Theorem}

\newcommand{\cali}[1]{\mathscr{#1}}

\usepackage{tikz}
\usepackage{multicol}
\usepackage{MnSymbol}
\usepackage{relsize}

\newcommand{\Curlywedge}{\mathlarger \curlywedge}

\newcommand{\diff}{\text{\normalfont d}}

\newcommand{\supp}{{\rm supp}}

\newcommand{\ddc}{\text{\normalfont dd}^c}
\newcommand{\dc}{\text{\normalfont d}^c}

\newcommand{\PSH}{{\rm PSH}}

\newcommand{\B}{\mathbb{B}}
\newcommand{\C}{\mathbb{C}}

\newcommand{\N}{\mathbb{N}}

\renewcommand\P{\mathbb{P}}

\title{\bf Intersection of $(1,1)$-currents and the domain of definition of the Monge-Amp\` ere operator}

\author{Dinh Tuan Huynh, Lucas Kaufmann and Duc-Viet Vu}

\newcommand{\Addresses}{{

		\footnotesize
		\noindent
		\textsc{Dinh Tuan Huynh, Hua Loo-Keng center for Mathematical Sciences, Academy of Mathematics and System Science, Chinese Academy of Sciences, Beijing 100190, China 
\& Department of Mathematics, University of Education, Hue University, 34 Le Loi St., Hue City, Vietnam }
		\par\nopagebreak
		\noindent
		\textit{E-mail address}: \texttt{dinhtuanhuynh@hueuni.edu.vn}
		\newline
		
		\noindent
		\textsc{Lucas Kaufmann, Department of Mathematics,  National University of Singapore - 10, Lower Kent Ridge Road - Singapore 119076. \textit{Current address:} Center for Complex Geometry - Institute for Basic Science (IBS) - 55 Expo-ro Yuseong-gu Daejeon 34126 South Korea}
		\noindent
		\par\nopagebreak
		\noindent
		\textit{E-mail address}: \texttt{lucaskaufmann@ibs.re.kr; lucaskaufmann.math@gmail.com}
		\newline		
		
		\noindent
		\textsc{Duc-Viet Vu, University of Cologne, Mathematical Institute, Weyertal 86-90, 50931, K\"oln,  Germany}
		\noindent
		\par\nopagebreak
		\noindent
		\textit{E-mail address}: \texttt{vuviet@math.uni-koeln.de}	
}}

\begin{document}

\begin{abstract}  
We study the Monge-Amp\` ere operator within the framework of Dinh-Sibony's intersection theory defined via density currents. We show that if $u$ is a plurisubharmonic function belonging to the B\l ocki-Cegrell class, then the Dinh-Sibony $n$-fold self-product of $\ddc u$ exists and coincides with the classically defined Monge-Amp\`ere measure $(\ddc u)^n$.
\end{abstract}

\maketitle

\noindent


\noindent



\vspace{-40pt}

\section{Introduction}

Let $\Omega$ be a domain in $\C^n$ and let $u$ be a plurisubharmonic  (p.s.h.\ for short) function on $\Omega$. A question of central importance in pluripotential theory and its applications is whether one can define the Monge-Amp\` ere measure $$(\ddc u)^n = \ddc u \wedge \cdots \wedge \ddc u$$ in a meaningful way. Recall that $\dc:=  \frac{i}{2\pi} (\overline \partial - \partial)$ and $\ddc = \frac{i}{\pi} \partial \bar \partial$.

For bounded p.s.h.\ functions, the definition of $(\ddc u)^n$ and the study of its fundamental properties are due to Bedford-Taylor \cite{Bedford_Taylor_76}. The problem of finding the largest class of p.s.h.\ functions where the Monge-Amp\` ere operator is suitably defined and continuous under decreasing sequences  was studied for a long time and a complete characterization of this class was finally achieved by Cegrell \cite{cegrell:def-MA} and B\l ocki \cite{blocki:domain-MA}. We denote this class by $\mathcal D(\Omega)$ and call it the \textit{B\l ocki-Cegrell class}. 

\vskip5pt

The question of defining $(\ddc u)^n = \ddc u \wedge \cdots \wedge \ddc u$ is an instance of the fundamental problem of intersection of currents. Indeed, if we set $T := \ddc u$, then $T$ is a positive closed $(1,1)$-current on $\Omega$ and $(\ddc u)^n$ is the self-intersection $T^n = T \wedge \cdots \wedge T$.  The intersection theory of currents has been quite well-developed thanks to the work of many authors. The case of bi-degree $(1,1)$-currents is more accessible due to the existence of p.s.h.\ functions as local potentials. For this reason, this case was soon developed, see \cite{Chern_Levine_Nirenberg,Bedford_Taylor_76,Fornaess_Sibony,Demailly_ag}. Later on, other notions of intersection were introduced, such as the non-pluripolar product \cite{BT_fine_87,BEGZ} (see also \cite{Viet-generalized-nonpluri} for recent developments) and  the  Andersson-Wulcan product of $(1,1)$-currents with analytic singularities \cite{andersson-wulcan}. All these generalized notions differ from  classical ones by the fact that, in one way or another,  one removes the singular set of the currents before intersecting them. As a drawback, there is a mass loss in this procedure.  

A general intersection theory for currents of higher bi-degree was developed only later. Most notably, Dinh-Sibony proposed two different notions of intersection, one using what they call superpotentials \cite{DinhSibony_Pk_superpotential} and, more recently, another one based on the notion of density currents, that we consider here. We refer to the orginal paper \cite{Dinh_Sibony_density}  and also \cite{viet:imrn} for generalizations and simplified arguments. Both approaches have already found many applications in dynamical systems and foliation theory. 

The main goal of the present paper is to study the Monge-Amp\`ere operator from the point of view of theory of density currents. We now briefly recall this notion. More details are given in Section \ref{sec:densities}.

Let $X$ be a complex manifold and let $T_1,\ldots,T_m$ be positive closed currents on $X$.  Consider the Cartesian product $X^m$ and the positive closed current $\mathbf T = T_1 \otimes \cdots \otimes T_m$  on $X^m$. Let $\Delta=\{(x,\dots,x):x\in X\} \subset X^m$ be the diagonal and $N\Delta$ be its normal bundle inside $X^m$. Using a certain type of local coordinates $\tau$ in $X^m$ around $\Delta$ with values in $N\Delta$, which are called admissible maps, we can consider the current $\tau_* \mathbf T$ defined around the zero section of $N \Delta$.

For  $\lambda \in \C^*$, let $A_\lambda: N\Delta \to N\Delta$ be the fiberwise multiplication by  $\lambda$. A \textbf{density current} $R$ associated with $(T_1,\ldots,T_m)$ is a positive closed current on $N \Delta$ such that there exists a sequence of complex numbers  $\{\lambda_k\}_{k \in \N}$ converging to $\infty$ for which $R = \lim_{k \to \infty} (A_{\lambda_k})_* \tau_* \mathbf T,$ for every admissible map $\tau$. We then say that the \textbf{Dinh-Sibony product} $T_1 \, \Curlywedge\,  \cdots \, \Curlywedge \, T_m$ of $T_1,\ldots,T_m$ exists if there is only one density current $R$  associated with $(T_1,\ldots,T_m)$ and $R = \pi^* S$ for some positive closed current on $\Delta$, where $\pi: N \Delta \to \Delta$ is the canonical projection. In that case we define $$T_1 \, \Curlywedge\,  \cdots \, \Curlywedge \, T_m:= S.$$

Our main result is the following, see Theorem \ref{thm:MA=DS} below.

\begin{main-theorem*}  \label{thm:MA=DS-intro}
Let $\Omega$  be a domain in $\C^n$ and let $u_1,u_2, \ldots, u_m$, $1 \leq m \leq n$ be plurisubharmonic functions in the B\l ocki-Cegrell class. Then, the Dinh-Sibony product of  $\ddc u_1, \ldots, \ddc u_m$ is well-defined and
\begin{equation} \label{eq-BlockicegreallDSintro}
\ddc u_1 \, \Curlywedge \, \ldots  \, \Curlywedge \,  \ddc u_m = \ddc u_1 \wedge \ldots \wedge \ddc u_m.
\end{equation}
In particular, for every  $u$ in the B\l ocki-Cegrell class, the operator $u \mapsto (\ddc u)^{\Curlywedge n}:=\ddc u\,\Curlywedge\ldots\,\Curlywedge\,\ddc u $ is well-defined and coincides with the usual Monge-Amp\`ere operator.
\end{main-theorem*}

The last conclusion of the above theorem says, in other words,  that the domain of definition of the Monge-Amp\`ere operator defined via Dinh-Sibony's product contains the B\l ocki-Cegrell class. The right hand side of (\ref{eq-BlockicegreallDSintro}) is obtained in the standard way, i.e.\ , by considering sequences of smooth p.s.h.\ functions decreasing to $u_j$, $j=1,\ldots,m$, similarly to the case where $m=n$ and all the $u_j$ are equal.  The fact that this mixed product is well-defined and independent of the chosen sequence is the content of Proposition \ref{prop:blocki-cegrell-current} below. Although not explicitly stated in the literature, this fact might be well-known among experts. The case $m=n$ can be found in \cite{cegrell:def-MA} and the case $m<n$ covered by Proposition \ref{prop:blocki-cegrell-current} follows from simple modifications of arguments from \cite{cegrell:def-MA} and \cite{blocki:domain-MA}.

Our main theorem is a strengthening of Theorem 1.1 in  \cite{Viet_Lucas} proven by the last two authors and it follows from a more general result, see Theorem \ref{the-cegrell-DS}.   This yields  an optimal result that covers the most general case where the classical Monge-Amp\`ere operator is well-defined and continuous with respect to decreasing sequences. 

 The proof of Theorem \ref{the-cegrell-DS} in the present work uses different techniques than the ones in  \cite{Viet_Lucas}, providing new and more clear arguments  In particular, the integrability assumption on the $u_j$'s required in \cite{Viet_Lucas} cannot be dropped with the techniques used there, so they cannot be applied in our situation. Here, we bypass this difficulty and show that those assumptions are actually unnecessary, yielding the optimal result. In order to achieve that, we obtain almost everywhere vanishing properties for Lelong numbers with respect to singular currents in the considered class (cf. Lemma \ref{le-dkisuyraii}). Also, by systematically working with a well-chosen set of test forms (cf. Definition \ref{def:split-forms}), we can easily get the compactness of the family of dilated currents (Lemma \ref{lemma:bounded-mass} below), which was overlooked in \cite{Viet_Lucas}. We refer to the end of this introduction and the comments after Theorem \ref{the-cegrell-DS} for an overview of the arguments and the new ingredients of the proof.

It is worth mentioning that the many notions of intersection quoted here are related to one another.  As mentioned before, the present paper together with \cite{Viet_Lucas} show that the intersection via density currents cover all known classical products of $(1,1)$-currents. For higher bi-degree currents it is also known that the density product generalizes the product of currents with continuous super-potentials, see \cite{DNV}.  Concerning generalized notions of products of $(1,1)$-currents, such as the non-pluripolar product and the Andersson-Wulcan product, some comparison results were obtained in \cite{Viet_Lucas}. The general phenomenon is that, in some sense, these products are dominated by the corresponding density currents. Moreover, the ideas of the present paper were further developed in \cite{Viet-density-nonpluripolar} to prove that the Dinh-Sibony product of $(1,1)$-currents of full mass intersection in K\"ahler classes on a compact K\"ahler manifolds exists and is equal to their (relative) non-pluripolar product.

\vskip5pt

To end this introduction, let us outline the structure and the key points of the proof of our main theorem. Our main result will follow from a more general theorem showing that if a mixed product is well-defined in the sense that it is obtained via decreasing sequences of smooth p.s.h.\ functions, then the corresponding Dinh-Sibony product exists and coincides with the classical one. This is the content of Theorem \ref{the-cegrell-DS} below. Its proof is obtained by induction on the number of currents involved and uses the following key facts: the vanishing of Lelong numbers with respect to the currents obtained in previous steps (Lemma \ref{le-dkisuyraii}), the interpretation of Lelong numbers as the mass of dilated currents (Lemma \ref{lemma:lelong-density}) and the observation that the dilation procedure in the definition of density currents yields canonical regularizations via decreasing sequences of p.s.h.\ functions (cf.\ the proof of Lemma {\ref{lemma:top-degree}). A more detailed outline is given below, after the statement of Theorem \ref{the-cegrell-DS}.  Finally, in Theorem \ref{thm:MA=DS}, we prove our main theorem by verifying that functions in the  B\l ocki-Cegrell  class satisfy the assumptions of Theorem  \ref{the-cegrell-DS}.

\vskip7pt

\textbf{Acknowledgments.} We would like to thank Tien-Cuong Dinh and Hoang-Son Do for fruitful discussions. D. T. Huynh is grateful to the Academy of Mathematics and System Sciences in Beijing for its hospitality and financial support. He also wants to acknowledge partial support from the Core Research Program of Hue University, Grant No. NCM.DHH.2020.15. L. Kaufmann was supported by grant R--146--000--259--114 from the National University of Singapore. D.-V. Vu is supported by a postdoctoral fellowship of the Alexander von Humboldt Foundation. Finally, we thank the anonymous referee for the careful reading of the first version of the manuscript and for the comments that helped us clarify our presentation.

\section{Preliminaries on density currents} \label{sec:densities}

In this section, we recall the definition and basic properties of tangent and density currents. For details, the reader is refered to the original paper \cite{Dinh_Sibony_density} and to \cite{Viet_Lucas}, \cite{viet:imrn}, \cite{DNV}  for more material.

Let $X$ be a complex manifold of dimension $n$ and $V$ be a smooth complex submanifold of $X$ of dimension $\ell$.  Let $T$ be a  positive closed $(p,p)$-current on $X$ with $0 \le p \le n$.  

Let $NV$ be the normal bundle of $V$ in $X$ and denote by $\pi: NV \to V$ the canonical projection. We identify $V$ with the zero section of $NV$. Let $U$ be an open subset of $X$ with $U \cap V \neq \varnothing$. An \textbf{admissible map} on $U$ is a smooth diffeomorphism $\tau$ from $U$ to an open neighbourhood of $V\cap U$ in $NV$ such that $\tau$ is the identity map on $V\cap U$ and the restriction of its differential $\diff \tau$ to $NV|_{V \cap U}$ is the identity.  Using a Hermitian metric on $X$, we can always find an admissible map defined  on a small tubular neighbourhood of $V$, see \cite[Lemma 4.2]{Dinh_Sibony_density}. This map is not holomorphic in general. However, if one only works on a small open set of $X$, it is easy to obtain holomorphic admissible maps.

For $\lambda \in \C^*,$ let $A_\lambda: NV \to NV$ be the multiplication by $\lambda$ along the fibers of $NV$. Consider the family of currents $(A_\lambda)_* \tau_* T$ on $NV|_{V \cap U}$ parametrized by $\lambda \in \C^*$. Following \cite{Dinh_Sibony_density,Viet_Lucas,viet:imrn}, we have:

\begin{definition}  \label{def:tangent-current}   A \textbf{tangent current}  of $T$ along $V$ is a positive closed current  $R$ on $NV$ such that there exist a sequence $(\lambda_k)_{k \geq 1}$ in $ \C^*$ converging to $\infty$ and a collection of holomorphic admissible maps $\tau_j: U_j \to NV$, $j \in J$ whose domains cover $V$ such that
$$R = \lim_{k \to \infty} (A_{\lambda_k})_* (\tau_j)_* T$$
on $\pi^{-1}(U_j \cap V)$ for every $j \in J$.
\end{definition}


Before continuing, let us make some comments on Definition \ref{def:tangent-current}. In \cite{Dinh_Sibony_density}, the authors considered the situation where $X$ is K\"ahler and  $\supp T \cap V$ is compact. There, a tangent current to $T$ along $V$ is defined as a limit current of the family $(A_{\lambda})_* \tau_* T$ as $|\lambda| \to \infty$, where $\tau$ is a \emph{global} admissible map, defined on an open tubular neighborhood of $V$; see \cite[Definition 4.5]{Dinh_Sibony_density}. Then, they proceed to show, cf.  \cite[Proposition 4.4]{Dinh_Sibony_density}, that tangent currents are independent of the choice of global admissible maps and can be localized in the following sense: if $U$ is an open subset of $X$ and $S=  \lim_{k \to \infty} (A_{\lambda_k})_*\tau_* T$, then  for every \emph{local} admissible map $\tau': U \to NV$ we have  $S=  \lim_{k \to \infty} (A_{\lambda_k})_*\tau'_* T$ on $\pi^{-1}(U \cap V)$. Therefore, our definition of tangent currents is equivalent to that of  \cite{Dinh_Sibony_density} when $\supp T \cap V$ is compact and $X$ is K\"ahler. Note also that, in this situation, it is shown in \cite{Dinh_Sibony_density} that tangent currents always exist. 

However,  in the cases we consider here $\supp T \cap V$ is not necessarily compact. Therefore, it is unclear whether we can use the original definition of  \cite{Dinh_Sibony_density}. This is because using only global admissible maps it is hard to ensure that the family  $(A_{\lambda})_* \tau_* T$ is compact and that the limit currents are independent of $\tau$. That is why we adopt a more flexible definition using local holomorphic admissible maps.

\vskip5pt
 
As in the compact setting, tangent currents depend in general on the sequence $(\lambda_k)_{k \geq 1}$.  The existence of tangent currents in the local setting is a more delicate matter and we have to prove it in our particular situation. However, if such currents exist, they are still independent of the choice of admissible maps.

\begin{lemma}\cite[Proposition 2.5]{Viet_Lucas} \label{le_globaladmiss}  Let $\tau: U \to NV$ be a holomorphic admissible map. Assume that there is a sequence $(\lambda_k)_{k \geq 1}$ tending to $\infty$ such that $ (A_{\lambda_k})_* \tau_* T$ converges to some current $R$ on $\pi^{-1}(U \cap V)$.   Then, for any other admissible map $\tau': U' \to NV$, we have 
$$R= \lim_{k \to \infty} (A_{\lambda_k})_* \tau'_* T$$
on $\pi^{-1}(U \cap U' \cap  V).$ 
\end{lemma}

 
A density current is a particular type of tangent current where $V$ is the diagonal inside a product space. More precisely, let $m \geq 1$ and let $T_j$ be positive closed $(p_j, p_j)$-currents for $1 \le j \le m$ on $X$. We usually assume that $p = p_1+\cdots+p_m \leq n$.  Let $\mathbf T = T_1 \otimes \cdots \otimes T_m$ be their tensor product. Then $\mathbf T$ is a positive closed $(p,p)$-current on $X^m.$ Let $\Delta = \{(x,\ldots,x): x \in X\} \subset X^m$ be the diagonal.  A \textbf{density current} associated with $T_1, \ldots,  T_m$ is a tangent current of $\mathbf T$ along $\Delta$. By definition, a tangent current is a positive closed $(p,p)$-current on the normal bundle $N\Delta$ of $\Delta$ inside $X^m$.

Let $\pi: N\Delta \to \Delta$ be the canonical projection. The following definition is given in \cite{Dinh_Sibony_density}.

\begin{definition} We say that the \textbf{Dinh-Sibony product} $T_1 \, \Curlywedge\,  \cdots \, \Curlywedge \, T_m$ of $T_1, \ldots, T_m$ exists if there is a unique density current $R$ associated with $T_1, \ldots, T_m$ and $R = \pi^* S$ for some current $S$ on $\Delta  = X$. In this case we define $$T_1 \, \Curlywedge\,  \cdots \, \Curlywedge \, T_m:= S.$$
\end{definition}

\section{Dinh-Sibony product and classical products} \label{sec-DSproduct}

Let $\Omega$ be a domain in $\C^n$. For a p.s.h.\ function $u$ on $\Omega$ and a point $x \in \Omega$, we denote by $\nu(u, x)$ the \textit{Lelong number} of $u$ at $x$.  See \cite{Demailly_ag} for various definitions and properties of the Lelong number.

The aim of this section is to prove the following general result. Our main theorem will be a consequence of it.

\begin{theorem} \label{the-cegrell-DS} Let $m \geq 2$ and $p \geq 0$ be such that $m-1+p \leq n$. Let $u_1, \ldots, u_{m-1}$ be p.s.h.\ functions on $\Omega$ and let $T$ be a positive closed $(p,p)$-current on $\Omega$.  Assume that for every subset $J=\{j_1, \ldots, j_{k}\} \subset  \{1, \ldots, m-1\}$, there is a current $R_J$ on $\Omega$ so that, for  any open set $U \subset \Omega$, for $j \in J$ and any sequence of smooth p.s.h.\ functions $(u^\ell_{j})_{\ell \in \N}$ decreasing to $u_j$ on $U$ as $\ell \to \infty$, one has 
\begin{equation} \label{eq:def-R_J}
\ddc u^\ell_{j_1} \wedge \cdots \wedge \ddc u^\ell_{j_{k}} \wedge T \longrightarrow R_{J} \quad \text{on } \, U \,\,\text{as } \,\ell \to \infty.
\end{equation}
We then define $\ddc u_{j_1} \wedge \cdots \wedge \ddc u_{j_{k}} \wedge T$ as the current $R_J$. If $J= \varnothing$, we set $R_{J}:= T$.

 


\vskip5pt

Then, the Dinh-Sibony product of  $\ddc u_1, \ldots, \ddc u_{m-1}, T$ is well-defined and one has 
\begin{equation} \label{eq:cegrell-DS}
\ddc u_1 \curlywedge \cdots \curlywedge \ddc u_{m-1} \curlywedge T = R_{\{1,\ldots, m-1\}} = \ddc u_1 \wedge \cdots \wedge \ddc u_{m-1} \wedge T .
\end{equation}

\end{theorem}

\vskip5pt

Before stating the preparatory results, let us briefly outline the structure of the proof. From the definition of density product we have to show that $R_{1,\lambda} \to \pi^* R_1$ as $\lambda \to \infty$, where $R_{1,\lambda}$ is the dilation of the tensor product of $T$ and $\ddc u_j$, $j=,\ldots, m-1$ along the diagonal. We will argue by induction on $m$. After fixing a suitable coordinate system on $(\C^n)^m = (\C^n,y^1) \times \cdots \times (\C^n,y^m)$ it is enough to work with two types of test forms $\Phi$. For each type, the estimates are of a different nature.
\begin{itemize}
\item  \textit{forms of type} $\rm I$: $\Phi = \phi_1(y^1) \wedge \phi_2(y^2) \wedge \ldots  \wedge \phi_m(y^m)$,
where $\phi_j$ are positive $(p_j,p_j)$-forms on $(\C^n,y^j)$ and at least one among $\phi_1,\ldots,\phi_{m-1}$ is not of top degree. In this case we have that $\langle   R_{1,\lambda}, \Phi \rangle \to 0$ as $\lambda \to \infty$, see Lemma \ref{lemma:not-top-deg}. Here we use Lemma \ref{le-dkisuyraii}, saying that the Lelong numbers of $u_j$ are negligible with respect to the currents obtained in previous steps, and a characterization of Lelong numbers in terms of dilated currents (Lemma \ref{lemma:lelong-density}). 

\medskip

\item  \textit{forms of type} $\rm II$: $\Phi = \phi_1(y^1) \wedge \phi_2(y^2) \wedge \ldots  \wedge \phi_m(y^m)$,
where $\phi_j$ is a radial $(n,n)$- form for every $j=1,\ldots, m-1$. In this case $\langle   R_{1,\lambda}, \Phi \rangle \to \langle R_1, \pi_*\Phi \rangle \quad \text{as } \, \,  \lambda \to \infty$, see Lemma \ref{lemma:top-degree}. Here we use that the sequence of dilations yields canonical regularizations via decreasing sequences of p.s.h.\ functions. 
\end{itemize}

\vskip5pt

Working with forms of type $\rm I$ allows us to prove that the limit currents have minimal horizontal dimension and, therefore, are the pullback of some current  $R^h_{1,\infty}$ in the diagonal, cf. Lemma \ref{lemma:minimal-h-dim}. On the other hand, working with forms of type $\rm II$  let us recognize the current $R^h_{1,\infty}$ as being $R_1$.

\vskip5pt

We now present the auxiliary lemmas we'll need. The proof of Theorem \ref{the-cegrell-DS} is given in the end of this section.

\begin{lemma} \label{le-dkisuyraii} Let the notations and the hypothesis be as in Theorem \ref{the-cegrell-DS}. Then  for every $J \subset \{1, \ldots, m-1\} $ and every $1\leq k \leq m-1$ such that $k \not \in J$ we have that $\nu(u_{k}, \cdot)= 0$ almost everywhere with respect to the trace measure of $R_{J}$.
\end{lemma}

\proof  We work locally. Let  $J \subset  \{1, \ldots, m-1\}$.  Let $(u^\ell_{j})_{\ell \in \N}$ be a sequence of smooth p.s.h.\ functions decreasing to $u_j$ as $\ell \to \infty$ for $j \in J$.  Let $k \in \{1, \ldots, m-1\} \backslash J$. 
Let $(u^\ell_k)_{\ell \in \N}$ be a sequence of locally bounded p.s.h.\ functions decreasing to $u_k$.  We claim that 
\begin{align}\label{conver-motsuyhai}
\ddc u^\ell_k \wedge R_J \to R_{J \cup \{k\}} \quad \text{as } \,\, \ell \to \infty.
\end{align}
Consider first the case where $u^\ell_k$ is smooth.  Let $\Phi$ be a test form with compact support and $\epsilon>0$ a constant. Using (\ref{eq:def-R_J})  and the fact that $u^{\ell}_k$ is smooth we can find, for each $\ell \geq 1$ an index $s_\ell$ satisfying  
\begin{align}\label{ine-haisuymot}
\big |\langle \ddc u^{\ell}_k \wedge R_J -\ddc u^{\ell}_k \wedge \bigwedge_{j \in J} \ddc u^{s_\ell}_j \wedge T, \Phi \rangle \big | = \big |\langle  R_J -  \bigwedge_{j \in J} \ddc u^{s_\ell}_j \wedge T, \ddc u^{\ell}_k  \wedge \Phi \rangle \big |  \le \epsilon.
\end{align}
We can choose $(s_\ell)_\ell$ so that it increases to $\infty$ as $\ell \to \infty$. Hence, $u^{s_\ell}_j$ decreases to $u_j$ for $j \in J$ and by  hypothesis, one obtains
$$\ddc u^{\ell}_k \wedge \wedge_{j \in J} \ddc u^{s_\ell}_j \wedge T \to R_{J \cup \{k\}}$$
as $s \to \infty$. It follows that 
$$\big |  \langle  \ddc u^{\ell}_k \wedge \bigwedge_{j \in J} \ddc u^{s_\ell}_j \wedge T - R_{J \cup \{k\}}, \Phi \rangle \big | \le \epsilon$$
for $\ell$ is big enough. 
Combining this with (\ref{ine-haisuymot}) gives  
$$\big | \langle \ddc u^{\ell}_k \wedge R_J - R_{J \cup \{k\}}, \Phi \rangle \big | \le 2\epsilon$$
for $\ell$ big enough.   Therefore,  (\ref{conver-motsuyhai}) follows if $u^\ell_k$ is smooth.

The case of general $u^{\ell}_k$ follows from a regularization argument.  Let $u^{\ell,\delta}_k$ be a standard smooth regularization of $u^{\ell}_k$ obtaining from convolution against a smoothing kernel, so that $u^{\ell,\delta}_k$  decreases to $u^{\ell}_k$ as $\delta \to 0$. For each $\ell$, let $\delta_\ell$ be small enough such that 
\begin{align}\label{ine-iiraiuellepsilon}    
\langle \ddc u^{\ell}_k \wedge R_J - \ddc u^{\ell, \delta_\ell}_k \wedge R_J, \Phi \rangle | \le \epsilon.
    \end{align}
We can choose $\delta_\ell$ to be decreasing in $\ell$. Hence, $u^{\ell, \delta_\ell}_k$ are smooth p.s.h.\ functions decreasing to $u_k$ as $\ell \to \infty$. By the first part of the proof, we see that $ \ddc u^{\ell, \delta_\ell}_k \wedge R_J$ converges to $R_{J \cup \{k\}}$.  This combined with (\ref{ine-iiraiuellepsilon}) gives $$\langle \ddc u^{\ell}_k \wedge R_J - R_{J\cup \{k\}}, \Phi \rangle | \le 2\epsilon$$
for $\ell$ big enough. Hence,  (\ref{conver-motsuyhai}) follows.

Recall that our goal is  to prove that $R_J$ has no mass on $\{\nu(u_k,\cdot)>0\}$.   Let $w(x)= \|x\|^2$, where $x$ is the standard coordinate system on $\C^n$. Let $N$ be a large constant and set 
$$u^\ell_k:= \log (e^{u_k}+ 1/\ell e^{N w}).$$

Then, the $u^\ell_k$ are locally bounded p.s.h.\ functions that decrease to $u$ as $\ell \to \infty$.  Suppose that $R_J$ has positive mass on $V:=\{\nu(u_k, \cdot)>0\}$, that is $\mathbf 1_V R_J \neq 0$. Notice that $u_k= -\infty$ on $V$ implies that $u^\ell_k = N w - \log \ell$ on $V$. It follows that 
$$\ddc u^\ell_k \wedge R_J \ge \ddc u^\ell_k \wedge (\bold{1}_V R_J)  = \ddc (u^\ell_k \bold{1}_V R_J)= N \ddc  w \wedge (\bold{1}_V R_J).$$

Let $K$ a fixed compact set that is charged by $\bold{1}_V R_J$. Then, the mass of $N \ddc  w \wedge (\bold{1}_V R_J)$ over $K$ equals $c N$ for some constant $c > 0$ independent of $\ell$. By  the above inequality and (\ref{conver-motsuyhai}) one gets that the mass of $R_{J \cup \{k\}}$ on $K$ is $ \geq c N$. Choosing $N$ large enough gives a contradiction. This finishes the proof.
\endproof

Let $(x^1, \ldots,x^m)$ be the canonical coordinate system in $\Omega^m$ and $\Delta$ be the diagonal of $\Omega^m$. Put $y^j:= x^j- x^m$ for $1 \le j \le m-1$ and $y^m:= x^m$. Then, $(y^1,\ldots, y^{m-1}, y^m)$ forms a new coordinate system on $\Omega^m$     and $\Delta= \{y^j= 0: 1 \le j \le m-1\}$ which is identified with $\Omega$. Using these coordinates, we identify naturally the normal bundle of $\Delta$ with the trivial bundle $\pi: (\C^n)^{m-1} \times \Omega \to \Omega$. Observe that the change of coordinates $\varrho: \Omega^m \to (\C^n)^{m-1} \times \Omega$ given by
\begin{equation}
\label{eq:standard-tau}
\varrho(x^1, \ldots, x^m) = (x^1-x^m, \ldots, x^{m-1} -x^m, x^m) :=(y^1, \ldots, y^m) := (y',y^m):= y
\end{equation}
 is a holomorphic admissible map. By Lemma \ref{le_globaladmiss}, it will be enough to work only with $\varrho$.

For $1 \le  j \le m-1,$ let  $T_j:=\ddc u_j$, $\widetilde T := \pi^* T $ and
$$\widetilde u_j(y',y^m) := \varrho_* u_j (y',y^m) = u_j(y^j + y^m).$$
We can check that $\widetilde{u}_j$ is locally integrable with respect to   $\ddc \widetilde u_{j+1} \wedge \cdots \wedge \ddc \widetilde u_{m-1} \wedge \widetilde T$ for $j=m-1, \ldots, 1$ and for every sequence $(u_{j}^\ell)_{\ell \in \N}$ of smooth p.s.h.\ functions decreasing to $u_j$ and $\widetilde u_{j}^\ell:= \varrho_* u_{j}^\ell$, we have 
\begin{align} \label{converge-ujl-rhopush}
\ddc \widetilde  u_{1}^\ell \wedge \cdots \wedge \ddc \widetilde  u_{m-1}^\ell \wedge \widetilde T \to  \ddc \widetilde u_1 \wedge \cdots \wedge \ddc \widetilde u_{m-1} \wedge \widetilde T
\end{align}
as $\ell \to \infty$. For the meaning of the right-hand side, see Definition \ref{def-bedford-taylor} below. The above assertions follow from a reasoning similar to the one in \cite[Lemma 2.3]{Viet_Lucas}.  Consequently, we get 
 \begin{equation} \label{eq-pushbyvarrhoTjtensor}
\varrho_* (T_1 \otimes \cdots \otimes T_{m-1}\otimes T) = \ddc \widetilde u_1 \wedge \cdots \wedge \ddc \widetilde u_{m-1} \wedge \widetilde T. 
\end{equation}

Now, for $1\le j\le m-1$,  put
\begin{equation*} \label{eq:def-Rjlambda}
R_{j,\lambda}:= (A_\lambda)_*\varrho_* (T_j \otimes \cdots \otimes T_{m-1}\otimes T) = (A_\lambda)_* \big( \ddc \widetilde u_j \wedge \cdots \wedge \ddc \widetilde u_{m-1} \wedge \widetilde T \big)
\end{equation*}
 and for $J \subset \{1,\ldots, m-1\}$, put
\begin{equation} \label{eq:def-RJlambda}
R_{J,\lambda}:= (A_\lambda)_* \bigg(  \bigwedge_{j \in J} \ddc \widetilde u_j  \wedge \widetilde T \bigg).
\end{equation}

Define also $$R_j := R_{\{j,\ldots,m-1\}} = \ddc  u_j \wedge \cdots \wedge \ddc  u_{m-1} \wedge  T.$$
We need to show that
\begin{equation*}
R_{1,\lambda} \,\, \stackrel{|\lambda| \to \infty}{\longrightarrow} \pi^* R_1. 
\end{equation*}
 We will do that by testing $R_{1,\lambda}$ against forms of different types.
 
 For future use, we note that
 \begin{equation} \label{eq:R1-lambda}
R_{1,\lambda} = \ddc u_{1}(\lambda^{-1}y^1 + y^m) \wedge \ldots \wedge  \ddc u_{m-1}(\lambda^{-1}y^{m-1} + y^m) \wedge T(y^m). 
\end{equation}
This is clear when the $u_j$ are smooth and the general case follows by regularizing the $u_j$ and using (\ref{converge-ujl-rhopush}).


In the definition below and throughout this paper, $i^n \diff y^j \wedge \diff \overline{y^j}$ will be a shorthand notation for the standard volume form on $(\C^n,y^j)$, that is $$i^n \diff y^j \wedge \diff \overline{y^j} := \Big( \sum_{k=1}^n i \diff y_k^j \wedge \diff \overline{y_k^j} \Big)^n.$$

\begin{definition} \label{def:split-forms}
Let $\Phi$ be a differential form on $(\C^n,y^1) \times (\C^n,y^2) \times \cdots \times (\C^n,y^m)$.  We say that $\Phi$ is a \textbf{positive split form} if it can be written as
\begin{equation*}
 \Phi = \phi_1(y^1) \wedge \phi_2(y^2) \wedge \ldots  \wedge \phi_m(y^m),
\end{equation*}
where $\phi_j$ are positive $(p_j,p_j)$-forms on $(\C^n,y^j)$.

An $(n,n)$-form $\phi_j$ on $(\C^n,y^j)$ is \textbf{radial} if it's rotation invariant, namely, if it is of the form $$\phi_j(y^j) = \chi(\|y^j\|^2) \cdot i^{n} \diff y^j \wedge \diff \overline{y^j}$$
 for some smooth function $\chi$.
\end{definition}

 In the sequel we will need the following expression of the Lelong number in terms of tangent currents. We denote by $\beta$ the standard K\"ahler form of $\C^n$ and by $B_\rho(0)$ the open ball of radius $\rho$ center at the origin in $\C^n$.

\begin{lemma} \label{lemma:lelong-density}
Let $S$ be a positive closed $(p,p)$-current defined near the origin in  $\C^n$. Let $A_\lambda(z) =  \lambda z$ and  set $S_\lambda:= (A_\lambda)_*S$. Let $\lambda_k$ be an increasing sequence tending to $\infty$ such that $S_{\lambda_k}$ converges to $S_\infty$. Let $\sigma_{S_\infty} = S_\infty \wedge \frac{1}{(n-p)!} \beta^{n-p}$ be the trace measure of $S_\infty$ and $\nu(S;0)$ be the Lelong number of $S$ at the origin. Then, there is a constant $c_p> 0$ depending only on $p$ such that  $\nu(S;0) = \lim_{\lambda_k \to \infty} c_p  \sigma_{S_{\lambda_k}}(B_1(0)) = c_p \,  \sigma_{S_\infty}(B_1(0))$ and the limit is decreasing.
\end{lemma}

\begin{proof}
To simplify the notation we may assume that the limit of $(A_{\lambda})_* S$ as $\lambda$ tends to infinity exists and is equal to $S_\infty$. Set $c_p:= \frac{(n-p)!}{\pi^{n-p}}$. Then, by the definition of Lelong number \cite{Demailly_ag} we have
\begin{align*}
\nu(S;0) &= \lim_{r \to 0} \frac{1}{ \pi^{n-p} \, r^{2n-2p}}  \int_{B_r(0)} S \wedge \beta^{n-p} = \lim_{|\lambda| \to \infty} \frac{1}{\pi^{n-p}} \, |\lambda|^{2n-2p} \int_{B_{1 \slash |\lambda|}(0)} S \wedge \beta^{n-p} \\ &= \lim_{|\lambda| \to \infty}  \frac{1}{\pi^{n-p}}  \int_{B_{1 \slash |\lambda|}(0)} S \wedge (A_\lambda)^*\beta^{n-p} = \lim_{|\lambda| \to \infty}  \frac{1}{\pi^{n-p}} \int_{B_{1 \slash |\lambda|}(0)} (A_\lambda)^* [ (A_\lambda)_*S \wedge \beta^{n-p}] \\  &= \lim_{|\lambda| \to \infty}  \frac{1}{\pi^{n-p}}  \int_{B_{1}(0)} (A_\lambda)_*S \wedge \beta^{n-p} \geq \frac{1}{\pi^{n-p}} \int_{B_{1}(0)} S_\infty \wedge \beta^{n-p} = c_p \,  \sigma_{S_\infty}(B_1(0)).
\end{align*}
In the second equality we have used that $A_\lambda^* \beta = |\lambda|^2 \,\beta$ and in the last inequality we have used the fact that if a sequence of measures $m_\lambda$ converges to $m$ and $U$ is open, then $\liminf_\lambda m_\lambda(U) \geq m(U)$. Hence 
\begin{equation} \label{eq:lelong-geq-T-infty}
\nu(S;0) \geq  c_p \,  \sigma_{S_\infty}(B_1(0)).
\end{equation}

Repeating the above argument on closed balls and using the fact that if a sequence of measures $m_\lambda$ converges to $m$  then $\limsup_\lambda m_\lambda(K) \leq m(K)$ for all closed sets $K$, we get that $\nu(S;0) \leq  c_p \,  \sigma_{S_\infty}(\overline{B_1(0)}).$ Now, the current $S_\infty$ is invariant by $(A_t)_*$ for every $t \in \C^*$ (see \cite{Dinh_Sibony_density}), hence its mass is homogeneous, namely $\sigma_{S_\infty}(B_\rho(0)) = \rho^{2n-2p} \sigma_{S_\infty}(B_1(0))$ for every $\rho > 0$ and similarly for the closed ball.  For $0 < \rho < 1$ this gives $$\nu(S;0) \leq c_p \,  \sigma_{S_\infty}(\overline{B_1(0)}) = c_p \, \rho^{2p-2n} \sigma_{S_\infty}(\overline{B_\rho(0)}) \leq  c_p \, \rho^{2p-2n} \sigma_{S_\infty}(B_1(0)).$$

Letting $\rho \nearrow 1$ gives $\nu(S;0) \leq  c_p \sigma_{S_\infty}(B_1(0))$. Together with (\ref{eq:lelong-geq-T-infty}) this gives the desired result.

It is a standard fact that all the above limits are decreasing as $r \to 0$, or equivalently as $\lambda \to \infty$, see \cite[III.5]{Demailly_ag}.
\end{proof}

Theorem \ref{the-cegrell-DS} will be proved by induction on $m$. The induction step will make use of next lemma. Let $u_1,\ldots,u_{m-1}$ and $T$ be as in Theorem  \ref{the-cegrell-DS}. For $J \subset \{1,\ldots,m-1\}$, let $R_{J,\lambda}$ be the current defined in (\ref{eq:def-RJlambda}) and $R_J = \bigwedge_{j \in J} \ddc u_j \wedge T$, defined as in (i) of Theorem  \ref{the-cegrell-DS}.

\begin{lemma} \label{lemma:not-top-deg}
With the above notation and the hypothesis of Theorem \ref{the-cegrell-DS}, assume that $R_{J,\lambda} \to \pi^* R_J$ as $\lambda \to \infty$ for every $J \subset \{1,\ldots,m-1\}$ such that $|J| \leq m -2$.

Let $\Phi$ be a positive split test form with compact support on  $(\C^n,y^1) \times \cdots \times (\C^n,y^m)$. Assume that $\Phi$ is not of bi-degree $(n,n)$ on $y^k$ for some $ 1 \leq k \leq m-1$. Then  $\langle   R_{1,\lambda}, \Phi \rangle \to 0$ as $\lambda \to \infty$.
\end{lemma}

\begin{proof}
By assumption
\begin{equation*}
 \Phi = \phi_1(y^1) \wedge \ldots \wedge \phi_{m-1}(y^{m-1})  \wedge \phi_m(y^m),
\end{equation*}
where each $\phi_j$ is positive and compactly supported.   Observe that we only need to consider forms $\Phi$ such that $R_{1,\lambda} \wedge \Phi$ has full degree, otherwise the last product vanishes and the result is trivial.

Notice that the current $R_{1,\lambda}$ has only terms of degree $0,1$ or $2$ on each $y^j$, $j=1,\ldots,m-1$.  Therefore, it suffices to consider the case where $\phi_j$ has bi-degree $(n-1,n-1)$ or $(n,n)$ for every $ j=1,\ldots,m-1$. Set $$J = \big\{j \in \{1,\ldots,m-1\} : \phi_j \,\, \text{ has bi-degree } \, (n,n) \big\}$$ and $$ K = \big\{k \in \{1,\ldots,m-1\} : \phi_k \,\, \text{ has bi-degree } \, (n-1,n-1) \big\}.$$

It follows from the assumption on $\Phi$ that $K$ is non-empty and $|J| \leq m-2$. Hence, by hypothesis
\begin{equation} \label{eq:RJ-lambda-conv}
\lim_{|\lambda| \to \infty} R_{J,\lambda} = \pi^* R_J.
\end{equation}

 Set $\phi_J = \bigwedge_{j \in J} \phi_j$ and $\phi_K = \bigwedge_{k \in K} \phi_k$. Since $J \cup K = \{1,\ldots,m-1\}$, we have that $$\Phi = \phi_J \wedge \phi_K \wedge \phi_m.$$

It follows from (\ref{eq:R1-lambda}) that
\begin{equation*}
R_{1,\lambda} \wedge \Phi = \bigwedge_{k \in K} \big( \ddc  u_k(\lambda^{-1} y^k + y^m) \big) \wedge \phi_K \wedge R_{J, \lambda} \wedge \phi_J \wedge \phi_m.
\end{equation*}

Since $R_{1,\lambda} \wedge \Phi$ is a current of top degree in  $(\C^n,y^1) \times \cdots \times (\C^n,y^m)$, it must have bi-degree $(n,n)$ on each $y^j$ (otherwise $R_{1,\lambda} \wedge \Phi =0$ and the lemma is trivial).  Hence, for  $j \in J$ only the derivatives of $u_j$ with respect to $y^m$ will contribute, while for $k \in K$, only the derivatives of $u_k$ with respect to $y^k$ will contribute. This gives 
\begin{equation} \label{eq:R1lambda-Phi}
R_{1,\lambda} \wedge \Phi = \bigwedge_{k \in K} \big( \ddc_{y^k} u_k(\lambda^{-1} y^k + y^m) \big) \wedge \phi_K \wedge R_{J, \lambda} \wedge \phi_J \wedge \phi_m.
\end{equation}

Here, the symbol $\ddc_{y^k}$  means that we only consider the (weak) derivatives with respect to the $y^k$ variables. The fact that the above wedge product is well-defined is obvious when the $u_j$ are smooth. This is less obvious for non-smooth functions, but it can be justified as in \cite[Lemma 2.3]{Viet_Lucas}. 
Now, for fixed $y^m$ and $k \in K$ we have that
\begin{equation*}
\Big|  \int_{y^k} \ddc_{y^k} u_k(\lambda^{-1}y^k + y^m) \wedge \phi_k(y^k) \Big|  \leq c_ k\int_{B_k} \ddc_{y^k} u_k(\lambda^{-1}y^k + y^m) \wedge \beta^{n-1}(y^k),
\end{equation*}
where $c_k > 0$ is a constant independent of $y^m$, $\beta$ is the standard K\"ahler form on $(\C^n,y^k)$ and $B_k$ is a ball in $(\C^n,y^k)$ containing the support of $\phi_k$. By Lemma \ref{lemma:lelong-density}, the integral on the right-hand side of the above inequality decreases to a constant independent of $y^m$ times the Lelong number of the $(1,1)$-current $\ddc_{y^k} u_k(y^k + y^m)$ at $y^k=0$, which is equal to $\nu(u_k,y^m)$. Here we use that the Lelong number of a positive closed $(1,1)$-current coincides with the Lelong number of any of its local potential (cf. \cite[III.6.9]{Demailly_ag}). Hence, for every $y^m$ one has
\begin{equation} \label{eq:limsup-lelong}
\limsup_{|\lambda| \to \infty} \Big|  \int_{y^k} \ddc_{y^k} u_k(\lambda^{-1}y^k + y^m) \wedge \phi_k(y^k) \Big| \lesssim \nu(u_k,y^m).
\end{equation}
Combining this with (\ref{eq:R1lambda-Phi}), the hypothesis that $R_{J, \lambda} \to \pi^* R_J$ as $\lambda \to \infty$ and Lemma \ref{lemma:lim-sup-usc} below, one obtains
\begin{equation*} \label{eq:R1lambda-Phi-2}
\limsup_{|\lambda| \to \infty} | \langle R_{1,\lambda} \wedge \Phi \rangle |  \lesssim  \int_{y^m} \bigg( \prod_{k \in K} \nu(u_k,y^m) \bigg)  R_J \wedge \phi_m.
\end{equation*}
The last integral in the above inequality vanishes because, by Lemma \ref{le-dkisuyraii}, $\nu(u_k,\cdot) = 0$ almost everywhere with respect to $R_J$. Therefore $$\limsup_{|\lambda| \to \infty} | \langle R_{1,\lambda} \wedge \Phi \rangle | = 0,$$ concluding the proof of the Lemma.
\end{proof}

We have used the following well known result.
\begin{lemma} \label{lemma:lim-sup-usc}
Let $X$ be a locally compact Hausdorff space. Let $m_\lambda$ be a sequence of Radon measures on $X$ whose supports are contained in a fixed compact subset of $X$. Assume  that  $m_\lambda \to m$ as $\lambda \to \infty$. Then for any sequence $(f_\lambda)_\lambda$ of continuous functions decreasing pointwise to a function $f$, we have that $$\limsup_{\lambda \to \infty} \int_X f_\lambda \, \diff m_\lambda \leq \int_X f \, \diff m.$$
\end{lemma}

\begin{lemma} \label{lemma:top-degree}
Under the assumptions of Theorem \ref{the-cegrell-DS}, let $\Phi = \phi_1(y^1) \wedge \ldots \wedge \phi_{m-1}(y^{m-1})  \wedge \phi_m(y^m)$ be a positive split test form with compact support on $(\C^n,y^1) \times \cdots \times (\C^n,y^m)$. Assume that  $\phi_j$ is a radial $(n,n)$- form for every $j=1,\ldots, m-1$. Then $$\langle   R_{1,\lambda}, \Phi \rangle \to \langle R_1, \pi_*\Phi \rangle \quad \text{as } \, \,  \lambda \to \infty.$$
\end{lemma}

\begin{proof}
After multiplying $\Phi$ by a positive constant, we can assume that $\int_{(\C^n,y^j)} \phi_j =1$ for every $j=1,\ldots, m-1$. Notice that $\pi_* \Phi = \phi_m$ and $\phi_m$ has bidegree $(n-m-p+1,n-m-p+1)$.

For $j=1,\ldots,m-1$, define $$ u^\lambda_{j} (y^m) :=  \int_{(\C^n,y^j)} u_{j}(\lambda^{-1}y^j + y^m) \,\phi_j (y^j).$$

Observe that $u^\lambda_j$ is a convolution against a (radially symmetric) smoothing  kernel on a disc of radius $|\lambda|^{-1}$ centered at $y^m$. Hence $u^\lambda_j$ is a smooth p.s.h.\ function on $(\C^n,y^m)$ decreasing pointwise to $u_j(y^m)$ as $\lambda \to \infty$ (see \cite[I.4.18]{Demailly_ag}). By (\ref{eq:def-R_J}) we get that
\begin{equation} \label{eq:R1-conv}
\ddc u_1^\lambda  \wedge \ldots \wedge \ddc u_{m-1}^\lambda  \wedge T \xrightarrow{\lambda \to \infty} R_1.
\end{equation}

Recall from (\ref{eq:R1-lambda}) that
\begin{equation*}
R_{1,\lambda} = \ddc u_{1}(\lambda^{-1}y^1 + y^m) \wedge \ldots \wedge  \ddc u_{m-1}(\lambda^{-1}y^{m-1} + y^m) \wedge T(y^m) .
\end{equation*}
Using the fact that the bidegree of each $\phi_j$, $j=1,\ldots,m-1$ is maximal, one has $$\ddc u_{j}(\lambda^{-1}y^j + y^m) \wedge \phi_j  = \ddc_{y^m} u_{j}(\lambda^{-1}y^j + y^m) \wedge \phi_j \quad j=1,\cdots, m-1.$$ 
Hence,
\begin{align*}
& R_{1,\lambda} \wedge \Phi =  \\ &\ddc_{y^m} u_{1}(\lambda^{-1}y^1 + y^m) \wedge \phi_1 (y^1) \wedge \ldots \wedge  \ddc_{y^m} u_{m-1}(\lambda^{-1}y^{m-1} + y^m) \wedge \phi_{m-1} (y^{m-1}) \wedge T(y^m) \wedge \phi_m (y^m).
\end{align*}
Taking the integral of both sides of the above equality and using Fubini's Theorem, one obtains
\begin{align*}
\langle   R_{1,\lambda}, \Phi \rangle &= \int_{(\C^n,y^m)} \bigg( \Big( \int_{(\C^n,y^1)} \ddc_{y^m} u_1 (\lambda^{-1}y^1 + y^m) \wedge \phi_1 (y^1) \Big) \wedge \cdots \\
&\cdots \wedge \Big( \int_{(\C^n,y^{m-1})} \ddc_{y^m} u_{m-1} (\lambda^{-1}y^{m-1} + y^m) \wedge \phi_{m-1} (y^{m-1}) \Big) \bigg) \wedge T(y^m) \wedge \phi_m(y^m) \\
&= \int_{(\C^n,y^m)} \ddc u_1^\lambda (y^m) \wedge \ldots \wedge\ddc u_{m-1}^\lambda (y^m) \wedge T(y^m) \wedge \phi_m(y^m) \\
&= \langle\ddc u_1^\lambda  \wedge \ldots \wedge\ddc u_{m-1}^\lambda  \wedge T, \phi_m\rangle.
\end{align*}

By (\ref{eq:R1-conv}), the last quantity tends to $\langle R_1, \phi_m \rangle = \langle R_1, \pi_*\Phi \rangle$ as $\lambda \to \infty$. This finishes the proof.
\end{proof}

The following result is an important consequence of the previous lemmas.

\begin{lemma} \label{lemma:bounded-mass}
Under the assumptions of Lemma \ref{lemma:not-top-deg}, the mass of $R_{1,\lambda}$ on compact sets is uniformly bounded.
\end{lemma}

\begin{proof}
Let $\omega:=  \sum_{k=1}^n i \diff y_k^j \wedge \diff \overline{y_k^j}$ be the standard K\"ahler form on $(\C^n)^m = (\C^n,y^1) \times \cdots \times (\C^n,y^m)$ and set $\Theta := \omega^{nm-m+1-p}$. In order to prove the desired assertion, using the fact that $R_{1,\lambda}$ is positive,  it is enough to check that the mass of the trace measure $R_{1,\lambda}\wedge \Theta$ is uniformly bounded on compact subsets of $\Omega^m$. 

Notice that  the form $\Theta$ is a linear combination of positive split forms. Therefore, in order to obtain the above bound, it will be enough to prove that $\langle   R_{1,\lambda}, \Phi \rangle $ is uniformly bounded for any fixed positive split test form $\Phi = \phi_1(y^1) \wedge \ldots \wedge \phi_{m-1}(y^{m-1})  \wedge \phi_m(y^m)$ with compact support.

If $\phi_j$  is not of top degree for some $j=1,\ldots,m-1$, then, by Lemma \ref{lemma:not-top-deg}, we have that  $|\langle R_{1,\lambda}, \Phi\rangle | \to 0$ as $\lambda \to \infty$. In particular $|\langle R_{1,\lambda}, \Phi\rangle |$ is uniformly bounded. Hence, we can assume that $\phi_j$ has bidegree $(n,n)$ for every $j=1,\ldots,m-1$. In this case, since $\phi_j$ is always bounded by some radial positive test form, we can assume furthermore that $\phi_j$ is radial for every $j$.   From this observation and Lemma \ref{lemma:top-degree}, we have $\langle R_{1,\lambda}, \Phi \rangle \to \langle R_1, \pi_*\Phi \rangle$ as $\lambda \to \infty$. In particular, $|\langle R_{1,\lambda}, \Phi\rangle |$ is uniformly bounded.  This finishes the proof of the Lemma.
\end{proof}

We now recall from \cite[Section 3]{Dinh_Sibony_density} the notion of horizontal dimension of currents on vector bundles. Actually,  the authors consider projective fibrations, that is, the projective compactification $\P(E)$ of a given holomorphic vector bundle $E$. Here, we phrase the definitions and results for  vector bundles instead. The proofs can be easily adapted from the ones in \cite{Dinh_Sibony_density}. 

Let $V$ be a K\"ahler manifold of dimension $\ell$ with K\"ahler form $\omega_V$ and let $\pi: E \to V$ be a holomorphic vector bundle over $V$.

\begin{definition}
Let $S$ be a non-zero positive closed current on $E$. The \textbf{horizontal dimension} ($h$-dimension for short) of $S$ is the largest integer $j$ such that $S \wedge \pi^* \omega_V^j \neq 0$.
\end{definition}

We will need the following characterization of currents of minimal $h$-dimension.

\begin{lemma} \label{lemma:shadow-minimal-h-dim}
Let $S$ be a positive closed $(p,p)$-current on $E$ with $p \leq \ell$. Assume that the $h$-dimension of $S$ is smaller or equal to $\ell - p$. Then  the $h$-dimension of  $S$ is equal to $\ell - p$ and there is a positive closed $(p, p)$-current $S^h$ on $V$ such that $S = \pi^*(S^h)$.
\end{lemma}
\begin{proof}
See \cite[Lemma 3.4]{Dinh_Sibony_density}.
\end{proof}

Now let $V = \Omega \subset (\C^n,y^m)$, $\omega_V = \sum_{k=1}^n \diff y^m_k \wedge \diff \overline y^m_k  := \beta(y^m)$ be the standard K\" ahler form on $\Omega$ and $E$ be the trivial bundle $\pi: (\C^n)^{m-1} \times \Omega \to \Omega$, $\pi(y',y^m) = y^m$. 

Recall from Lemma \ref{lemma:bounded-mass} that $(R_{1,\lambda})_\lambda$ is a relatively compact family of positive closed $(m-1+p,m-1+p)$-currents on $E$.

\begin{lemma} \label{lemma:minimal-h-dim}
In the assumptions of Lemma \ref{lemma:not-top-deg}, let $R_{1,\infty}$ be a limit point of the family $R_{1,\lambda}$ as $\lambda \to \infty$.  Then the $h$-dimension of $R_{1,\infty}$  is minimal, equal to $n-m+1-p$. In particular there is a positive closed $(m-1+p,m-1+p)$-current $R_{1,\infty}^h$ on $\Omega$ such that $R_{1,\infty} = \pi^* R_{1,\infty}^h$.
\end{lemma}

\begin{proof}
Let $\lambda_k$ be a sequence tending to $\infty$ such that $R_{1,\lambda_k} \to R_{1,\infty}$. By  Lemma \ref{lemma:shadow-minimal-h-dim}, we only need to show that $R_{1,\infty} \wedge \pi^* \beta^{n-m-p+2}(y^m) = 0$. To do this, it is enough to verify that $$\langle R_{1,\infty} \wedge \pi^* \beta^{n-m-p+2}(y^m), \Phi \rangle = 0$$ for every positive split test form $\Phi$.

Let $\Phi = \phi_1(y^1) \wedge \ldots \wedge \phi_{m-1}(y^{m-1})  \wedge \phi_m(y^m)$ be such a form. As in the beginning of the proof of Lemma \ref{lemma:not-top-deg}, we may assume that each $\phi_j$, $j=1,\cdots,m-1$ has bidegree $(n,n)$ or $(n-1,n-1)$. Since the total bidegree of $\Phi$ is  $(p',p')$, where 
$$p'= nm-(n-m-p+2)- (m-1+p)= nm-n-1,$$
at least one of $\phi_j$, $j=1,\ldots,m-1$ has bidegree $(n-1,n-1)$. In this case,  by Lemma \ref{lemma:not-top-deg}, one has  $\langle R_{1,\lambda} \wedge \pi^* \beta^{n-m-p+2}(y^m), \Phi \rangle \to 0$ as $\lambda \to \infty$.  This finishes the proof.
\end{proof}

We are now in position to prove Theorem \ref{the-cegrell-DS}. 

\begin{proof}[End of proof of Theorem \ref{the-cegrell-DS}] Recall our notation
\[
R_{j,\lambda}:= (A_\lambda)_* \big( \ddc \widetilde u_j \wedge \cdots \wedge \ddc \widetilde u_{m-1} \wedge \widetilde T \big) \eqno{{\scriptstyle( 1\,\leq\, j\, \leq\, m-1)},}
\]
\[
R_{J,\lambda}:= (A_\lambda)_* \bigg(  \bigwedge_{j \in J} \ddc \widetilde u_j  \wedge \widetilde T \bigg) \eqno{{\scriptstyle( J\, \subset\, \{1,\ldots, m-1\})},}
\]
\[
R_j := R_{\{j,\ldots,m-1\}} = \ddc  u_j \wedge \cdots \wedge \ddc  u_{m-1} \wedge  T\eqno{{\scriptstyle( 1\,\leq\, j\, \leq\, m-1)}.}
\]
Recall also that proving (\ref{eq:cegrell-DS}) is equivalent to proving that
\begin{equation*}
R_{1,\lambda} \,\, \stackrel{|\lambda| \to \infty}{\longrightarrow} \pi^* R_1. 
\end{equation*}

We'll proceed by  induction on $m$.  When $m=1$ the result is obvious.   Now let $m \geq 2$ and assume that  $R_{J,\lambda} \to \pi^* R_J$ as $\lambda \to \infty$ for every $J \subset \{1,\ldots,m-1\}$ such that $|J| \leq m -2$. When $m=2$ this assumption is vacuous.  Then, the hypothesis of Lemma \ref{lemma:not-top-deg} is satisfied. By Lemma \ref{lemma:bounded-mass}, the family $(R_{1,\lambda})_\lambda$ is relatively compact. Let $R_{1,\infty} = \lim_{\lambda_k \to \infty} R_{1,\lambda_k}$ be one of its limit points. By Lemma \ref{lemma:minimal-h-dim}, there is a positive closed $(m-1+p,m-1+p)$-current $R_{1,\infty}^h$ on $\Omega$ such that $R_{1,\infty} = \pi^* R_{1,\infty}^h$. We need to show that $R_{1,\infty}^h = R_1$.

Let $\phi_m$ be a test form on $(\C^n,y^m)$. Take $\phi_1,\ldots,\phi_{m-1}$ positive radial $(n,n)$-forms with compact support such that  $\int_{(\C^n,y^j)} \phi_j = 1$ for every $j=1, \ldots, m-1$. Then $\Phi := \phi_1  \wedge \cdots \wedge \phi_{m-1} \wedge \phi_m$ is such that $\pi_* \Phi = \phi_m$. Using Lemma \ref{lemma:top-degree}, we get 
\begin{align*}
\langle R_{1,\infty}^h, \phi_m \rangle &= \langle R_{1,\infty}^h, \pi_*\Phi \rangle = \langle \pi^* R_{1,\infty}^h, \Phi \rangle =  \langle  R_{1,\infty}, \Phi \rangle \\ &= \lim_{\lambda_k \to \infty} \langle  R_{1,\lambda_k}, \Phi \rangle =  \langle  R_{1}, \pi_*\Phi \rangle = \langle R_{1}, \phi_m \rangle.
\end{align*}
Since $\phi_m$ is arbitrary,  we get  $R_{1,\infty}^h = R_1$. This concludes the proof of the Theorem.
\end{proof}

\begin{remark}
In the statement of Theorem \ref{the-cegrell-DS}, if we consider the case where $m-1+p >n$, then  the arguments in  the above proof still work and we obtain that the associated density current vanishes.
\end{remark}

Let $R$ be a positive closed current and $v$ be a p.s.h.\ function. If $v$ is locally integrable with respect to (the trace measure) of $R$, we define, following Bedford-Taylor,
\begin{equation} \label{eq:def-bedford-taylor}
\ddc v \wedge R := \ddc (v R).
\end{equation}

For a collection $v_1, \ldots, v_{s}$ of  p.s.h.\ functions, we can apply the above definition recursively, as long as the integrability conditions are satisfied. 

\begin{definition} \label{def-bedford-taylor} We say that the intersection  of $\ddc v_1, \ldots, \ddc v_s,R$ is \emph{classically well-defined} if for every non-empty subset $J=\{j_1, \ldots, j_k\}$ of $\{1, \ldots, s\}$, we have that $v_{j_k}$ is locally integrable with respect to the trace measure of $R$ and inductively,  $v_{j_r}$ is locally integrable with respect to the trace measure of $\ddc v_{j_{r+1}} \wedge \cdots \wedge \ddc v_{j_k} \wedge R$ for $r=k-1, \ldots, 1$, and the product $\ddc v_{j_{1}} \wedge \cdots \wedge \ddc v_{j_k} \wedge R$ is continuous under decreasing sequences of p.s.h.\ functions.
\end{definition}

The last definition is slightly more restrictive than the one given in \cite{Viet_Lucas}.  We have the following comparison result between the Dinh-Sibony product and the above notion of  wedge products. This result is a direct consequence of  Theorem \ref{the-cegrell-DS}.


\begin{corollary} \label{cor:DS-vs-classical} Let $m \geq 2$ and $p \geq 0$ be such that $m-1+p \leq n$. Let $u_1, \ldots, u_{m-1}$ be p.s.h.\ functions on $\Omega$ and let $T$ be a positive closed $(p,p)$-current on $\Omega$.  Assume that $\ddc u_{1} \wedge \ldots \wedge \ddc u_{m-1} \wedge T$ is classically well-defined. Then the Dinh-Sibony product of  $\ddc u_1, \ldots, \ddc u_{m-1}, T$ is well-defined and 
$$\ddc u_1 \curlywedge \cdots \curlywedge \ddc u_{m-1} \curlywedge T =\ddc u_{1} \wedge \ldots \wedge \ddc u_{m-1} \wedge T.$$
\end{corollary}

We note that \cite[Theorem 1.1]{Viet_Lucas} asserts a similar conclusion, but there's a slip in the proof of the result as stated there.


\section{Products in the B\l ocki-Cegrell class and the domain of definition of the Monge-Amp\` ere operator}

In this section we apply Theorem \ref{the-cegrell-DS} to studying the domain of definition of the Monge-Amp\` ere operator via Dinh-Sibony's intersection product. Let $\Omega$ be a domain in $\mathbb{C}^n$. Denote by $\PSH(\Omega)$ the set of p.s.h.\ functions on $\Omega$.

\begin{definition} \label{def:MA}
The \textbf{B\l ocki-Cegrell class} on $\Omega$ is the subset $\mathcal D (\Omega)$ of $\PSH(\Omega)$ consisting of functions $u$ with the following property: there exists a measure $\mu$ in $\Omega$ such that for every open set $U \subset \Omega$ and every sequence $(u_\ell)_\ell$ of smooth p.s.h.\ functions on $U$ decreasing to $u$ pointwise as $\ell \to \infty$, we have that $(\ddc u_\ell)^n$ converges to $\mu$ as $\ell \to \infty$.

For $u \in \mathcal D (\Omega)$, we define $(\ddc u)^n := \mu$, where $\mu$ is the above measure.
\end{definition}

The class $\mathcal D (\Omega)$ is the largest subset of $\PSH(\Omega)$ where we can define a Monge-Amp\`ere operator that coincides with the usual one for smooth p.s.h.\ functions and which is continuous under decreasing sequences, see \cite{cegrell:def-MA,blocki:domain-MA}.

We first need the following result ensuring the existence of the mixed products in the B\l ocki-Cegrell class.

\begin{proposition} \label{prop:blocki-cegrell-current}
Let $\Omega$  be a domain in $\C^n$ and let  $u_1,u_2, \ldots, u_m$, $1 \leq m \leq n$ be p.s.h.\ functions in  $\mathcal D(\Omega)$. Then, there exists a positive closed $(m,m)$-current $S_m$ such that for every open set $U \subset \Omega$ and every sequence $(u^\ell_j)_\ell$ of smooth p.s.h.\ functions on $U$ decreasing to $u_j$ pointwise as $\ell \to \infty$, we have that
\begin{equation} \label{eq:blocki-cegrell-current-convergence}
\ddc u^\ell_1 \wedge \cdots \wedge \ddc u^\ell_m \to S_m \quad \text{on }\, U \,\text{ as }\, \ell \to \infty.
\end{equation}

\end{proposition}

 For $u_1,u_2, \ldots, u_m \in \mathcal D(\Omega)$, we define their wedge product by
\begin{equation}
\ddc u_1 \wedge \cdots \wedge \ddc u_m := S_m,
\end{equation}
where $S_m$ is the current  appearing in the above proposition. In particular, for $u \in \mathcal D(\Omega)$, one sees that $(\ddc u)^n$ is the Monge-Amp\`ere measure given in Definition \ref{def:MA}.

 As mentioned in the Introduction, Proposition \ref{prop:blocki-cegrell-current} is known when $m=n$ and the case $m<n$ might also be known to experts. We give a proof here for completeness, following closely the proof of \cite[Theorem 1.1]{blocki:domain-MA}. A simplifying step in \cite{blocki:domain-MA} is the fact that it suffices to work with test functions that are p.s.h.\ on a ball and vanish in its boundary. In the case $m<n$, this step is replaced by the following lemma.

\begin{lemma} \label{le-dense-hddcv}  Let $\B_1 \Subset \B_2 \Subset \Omega$ be balls. Let $\mathcal{A}$ be the vector space generated by forms of the type $h \, \ddc v_1 \wedge \cdots \wedge \ddc v_{n-m}$, where $h, v_1, \ldots,v_{n-m}$ are  p.s.h.\ functions on $\B_2$ which are continuous up to $\partial \B_2$ and vanish on $\partial \B_2$. Then, every  smooth $(n-m,n-m)$-form $\psi$ compactly supported in $\B_1$ is in $\mathcal{A}$. 
\end{lemma}

\proof  It is a standard fact that every smooth $(n-m,n-m)$-form $\psi$ compactly supported in $\B_1$ can be written as  a linear combination of   forms of type $\eta:=h \, i\gamma_1 \wedge \overline \gamma_1 \wedge \cdots \wedge i \gamma_{n-m} \wedge \overline \gamma_{n-m}$, where $h$ is a smooth function with compact support in $\B_1$ and $\gamma_1, \ldots, \gamma_{n-m}$ are $(1,0)$-forms  with constant coefficients, see  \cite[III.1.4]{Demailly_ag}. Hence, it is enough to prove the desired assertion for $\eta$ as above. 

Write $\gamma_\ell = \sum_{j=1}^n a_{j\ell} d z_j$, for $1 \le \ell \le n-m$, where  $a_{j \ell} \in \C$. Observe  $i \gamma_\ell \wedge \overline \gamma_\ell = \ddc v_\ell$, where $v_\ell(z):= \pi \big |\sum_{j=1}^n a_{j\ell} z_j \big|^2$, where $(z_1, \ldots, z_n)$ are the standard coordinates on $\C^n$. Let $\widetilde{v}_\ell$ be the envelope constructed from $v_\ell$ as in Lemma \ref{lemma:blocki-envelope} for $1 \le \ell \le n-m$. We have that  $\widetilde{v}_\ell= v_\ell$ on $\B_1$,  $\widetilde{v}_\ell \in \PSH(\B_2) \cap \cali{C}^0(\overline \B_2)$, 
 and $\widetilde{v}_\ell =0$ on $\partial \B_2$.  This combined with the fact that $h$ is compactly supported in $\B_1$ gives  $\eta= h \ddc \widetilde{v}_1 \wedge \cdots \wedge \ddc \widetilde{v}_{n-m}$. On the other hand, since $\B_2$ is a ball,  we can express $h= h_1 - h_2$ where $h_1, h_2$ are smooth p.s.h.\ functions such that $h_1= h_2=0$ on $\partial \B_2$. We deduce that $\eta \in \mathcal{A}$. This  finishes the proof.  
\endproof

For the proof of Proposition \ref{prop:blocki-cegrell-current}, we need the following result about Monge-Amp\`ere measures of envelopes. The first part is classical (see \cite{Bedford_Taylor_76,Walsh}), while the second part is contained in the proof of Theorem 1.1 in \cite{blocki:domain-MA}.

\begin{lemma} \label{lemma:blocki-envelope}
Let $\B_1 \Subset \B_2 \Subset  \Omega$ be balls compactly contained in $\Omega$. For a negative continuous function $v \in \PSH(\Omega)$, set $$ \widetilde{v} : = \sup \{  w \in \PSH(\B_2):  w <  v \,\text{ on }\, \B_1 \,\text{ and }\, w<  0 \, \text{ on } \, \B_2\}.$$

Then $\widetilde{v}$ is a p.s.h.\ function on $\B_2$ which is continuous on $\overline \B_2$ and satisfies
\begin{enumerate}[nosep]
\item  $\widetilde v = 0$ on $\partial \B_2$,
\item  $\widetilde v = v$ on $\overline{ \B}_1$,
\item $(\ddc \widetilde v)^n = 0 $ on $\B_2 \setminus \overline{ \B}_1$.
\end{enumerate}
\vskip3pt
Moreover, if $u \in \mathcal D(\Omega)$, then for any sequence $(u_\ell)_{\ell \geq 1}$  of smooth p.s.h.\ functions on $\Omega$ decreasing to $u$, we have $$\sup_{\ell \geq 1}\int_{\B_2} (\ddc \widetilde{u}_\ell)^n <+\infty.$$
\end{lemma}

\begin{proof}[Proof of Proposition \ref{prop:blocki-cegrell-current}]  Using Lemma \ref{le-dense-hddcv}, the proof is parallel to that of \cite[Theorem 1.1]{blocki:domain-MA}. We include the main differences in the argument for completeness. Since the problem is local, in order to get the desired assertion, it suffices to prove that  there exists a current $S_m$ on $\Omega$ such that   for every ball $\B_1 \Subset \Omega$ and  every sequence $(u^\ell_j)_{\ell \geq 1}$ of smooth p.s.h.\ functions on $\Omega$ decreasing to $u_j$ for $1 \le j \le m$,  we have  $\ddc u^\ell_1 \wedge \cdots \wedge \ddc u^\ell_m \to S_m$  on  $\B_1$ as  $ \ell \to \infty$.  

Let $\B_2 \Subset \Omega$ be  a ball containing $\overline \B_1$.   Let $h,v_1, \ldots, v_{n-m} \in \PSH(\B_2) \cap \cali{C}^0(\overline \B_2)$ be functions vanishing on  $\partial \B_2$. Put $\eta:=h \, \ddc v_1 \wedge \cdots \wedge \ddc v_{n-m}$. Let $\widetilde{u}_{j}$ be the envelope constructed from $u_{j}$ as in Lemma \ref{lemma:blocki-envelope} for $1 \le j \le m$. We have that 
\begin{align}\label{eq-vngabangvj42}
\widetilde{u}_{j}= u_j \quad \text{ on} \quad \B_1,
\end{align}
$\widetilde{u}_{j}$ is continuous up to $\partial \B_2$ and is equal to $0$ on $\partial \B_2$ for $1 \le j \le m$. Put 
$$S^\ell_m:= \ddc u^\ell_1 \wedge \cdots \wedge \ddc u^\ell_m, \quad  \widetilde{S}^\ell_m:= \ddc \widetilde{u}^\ell_1 \wedge \cdots \wedge \ddc \widetilde{u}^\ell_m.$$
We will prove that $\langle \widetilde S^\ell_m, \eta \rangle$ is convergent. By \cite[Corollary 5.6]{cegrell:def-MA}, 
we have 
\begin{align*}
\int_{\B_2} & \ddc \widetilde{u}^{\ell}_1 \wedge \cdots \wedge \ddc \widetilde{u}^{\ell}_{m}
\wedge \ddc v_1 \wedge \cdots \wedge \ddc v_{n-m}
\le  \\
& \bigg(\int_{\B_2} \big(\ddc \widetilde{u}^{\ell}_1\big)^n\bigg)^{1/n} \cdots \, \, \, \bigg(\int_{\B_2} \big(\ddc \widetilde{u}^{\ell}_m\big)^n\bigg)^{1/n}\notag \cdot \,\, \bigg(\int_{\B_2} \big(\ddc v_1\big)^n\bigg)^{1/n} \cdots \,\,\, \bigg(\int_{\B_2} \big(\ddc v_{n-m}\big)^n\bigg)^{1/n}.
\end{align*}
This combined with Lemma \ref{lemma:blocki-envelope} yields that $\langle \widetilde S^\ell_m, \eta \rangle$ is of uniformly bounded  as $\ell \to \infty$. With this last property, we can follow the exact same arguments from the proof of \cite[Theorem 1.1]{blocki:domain-MA}. This gives that $\lim_{\ell \to \infty}\langle \widetilde S^\ell_m, \eta \rangle$ exists and  is independent of the choice of the sequences $(u^\ell_j)_{\ell \geq 1}$.  Using this and Lemma \ref{le-dense-hddcv},  for every smooth form $\phi$ compactly supported in $\B_1$, we obtain that $\langle \widetilde S^\ell_m,\phi \rangle$ converges to a number independent of the choice of $(u^\ell_j)_{\ell \geq 1}$ as $\ell \to \infty$. On the other hand, by (\ref{eq-vngabangvj42}), we get 
$$\langle S^\ell_m,\phi \rangle=\langle \widetilde S^\ell_m,\phi \rangle.$$
Consequently,  the limit $\lim_{\ell \to \infty}\langle  S^\ell_m,\phi \rangle$ exists and is independent of the choice of $(u^\ell_j)_{\ell \geq 1}$ as $\ell \to \infty$. Hence, for $S_m$ defined by putting $\langle S_m, \phi\rangle:=\lim_{\ell \to \infty}\langle \widetilde S^\ell_m,\phi \rangle$ satisfies the desired property.    This concludes the proof of Proposition \ref{prop:blocki-cegrell-current}.
\end{proof}

\begin{theorem} \label{thm:MA=DS}
Let $\Omega$  be a domain in $\C^n$ and let $u_1,u_2, \ldots, u_m$, $1 \leq m \leq n$ be functions in  $\mathcal D(\Omega)$. Then, the Dinh-Sibony product of  $\ddc u_1, \cdots, \ddc u_m$ is well-defined and
\begin{equation} \label{eq-BlockicegreallDS}
\ddc u_1 \, \Curlywedge \, \ldots \, \Curlywedge \, \ddc u_m = \ddc u_1 \wedge \ldots \wedge \ddc u_m.
\end{equation}
In particular, for $u \in \mathcal D (\Omega)$, the Dinh-Sibony Monge-Amp\` ere operator $u \mapsto (\ddc u)^{\Curlywedge n} $ is well-defined and coincides with the usual one.
\end{theorem}

\begin{proof}
We'll apply Theorem \ref{the-cegrell-DS} to $u_1,u_2,\ldots,u_{m-1}$ and $T = \ddc u_m$. Let $J \subset  \{1, \ldots, m-1\}$. Then by Proposition \ref{prop:blocki-cegrell-current},  the current $R_J = \bigwedge_{j \in J} \ddc u_{j} \wedge \ddc u_m$ is well-defined. Now we check  the hypothesis of Theorem \ref{the-cegrell-DS} for $R_J$. Let $(u^\ell_j)_\ell$ be a sequence of smooth p.s.h.\ functions decreasing to $u_j$ for $j \in J$. We need to show that $\wedge_{j \in J} \ddc u^\ell_j \wedge \ddc u_m$ converges to $R_J$ as $\ell \to \infty$.  

Let $(u^\ell_m)_\ell$ be a sequence of smooth functions decreasing to $u_m$. Let $\Phi$ be a smooth test form with compact support and $\epsilon>0$ a constant. For every $\ell$, since $\ddc u^\ell_m\to \ddc u_m$, there exists $s_\ell \in \N$ such that 
\begin{align}\label{ine-MADS}
\big| \langle \wedge_{j \in J} \ddc u^\ell_j \wedge \ddc u^{s_\ell}_m- \wedge_{j \in J} \ddc u^\ell_j \wedge \ddc u_m, \Phi\rangle \big| \le \epsilon.
\end{align}
We can choose $s_\ell$ so that $s_\ell$ is decreasing in $\ell$. Hence, $u^{s_\ell}_m$ decreases to $u_m$. By  Proposition \ref{prop:blocki-cegrell-current}, we get   $\wedge_{j \in J} \ddc u^\ell_j \wedge \ddc u^{s_\ell}_m \to R_J$ as $\ell \to \infty$. This combined with (\ref{ine-MADS}) gives 
$$ \big| \langle \wedge_{j \in J} \ddc u^\ell_j \wedge \ddc u^{s_\ell}_m- R_J, \Phi\rangle \big| \le 2\epsilon$$
for $\ell$ big enough. Hence, $\wedge_{j \in J} \ddc u^\ell_j \wedge \ddc u_m$ converges to $R_J$ as $\ell \to \infty$. In other words, we have checked the  hypothesis of Theorem \ref{the-cegrell-DS} for $R_J$. The desired assertion follows.  The proof is finished.
\end{proof}

\bigskip

\noindent
\Addresses
\end{document}